\documentclass[letterpaper, 10 pt, conference]{ieeeconf}

\def\bsg{{\boldsymbol{g}}}

\def\bsv{{\boldsymbol{v}}}

\def\bsx{{\boldsymbol{x}}}

\def\bsH{{\boldsymbol{H}}}

\def\bsK{{\boldsymbol{K}}}
\def\bsL{{\boldsymbol{L}}}

\def\bsM{{\boldsymbol{M}}}

\def\bsQ{{\boldsymbol{Q}}}

\usepackage{geometry}
\usepackage{verbatim}
\usepackage{float}
\usepackage{bm}
\usepackage[cmex10]{amsmath}
\usepackage{amssymb}

\usepackage{graphicx}
\graphicspath{{./figures/}}

\usepackage{amsthm}

\usepackage{graphicx}
\usepackage{subcaption}
\usepackage{breakurl}
\usepackage{tikz}
\usepackage{verbatim}
\usepackage{forest}
\usepackage[T1]{fontenc}
\usepackage{pifont}
\usepackage{mathtools}
\usetikzlibrary{shapes,arrows,positioning,calc}
\usepackage{cite}
\usepackage[english]{babel}
\usepackage[utf8]{inputenc}

\usepackage{algorithm,algorithmic}

\newtheorem{theorem}{Theorem}
\newtheorem{assumption}{Assumption}

\newtheorem{lemma}{Lemma}

\newtheorem{definition}{Definition}
\newtheorem{remark}{Remark}

\usepackage{color}

\makeatletter

\tikzstyle{block} = [draw, thick, rectangle, 
minimum height=3em, minimum width=6em]
\tikzstyle{pt} = [coordinate]
\tikzset{
	block/.style={
		draw, 
		rectangle, 
		minimum height=0.8cm, 
		minimum width=0.8cm, 
		align=center
	}
}

\newcommand{\diag}{{\rm diag}}

\DeclareMathOperator{\col}{col}

\DeclareMathOperator{\Deg}{Deg}

\usetikzlibrary{arrows,shapes,positioning,shadows,trees}

\tikzset{
  basic/.style  = {draw, text width=5cm, drop shadow, font=\sffamily, rectangle},
  root/.style   = {basic, rounded corners=2pt, thin, align=center,
                   fill=yellow!60},
  level 2/.style = {basic, rounded corners=6pt, thin,align=center, fill=yellow!30,
                   text width=8em},
  level 3/.style = {basic, thin, align=left, fill=pink!40, text width=3cm}
}

\usepackage{adjustbox}
\usepackage{booktabs} 
\usepackage{makecell}

\allowdisplaybreaks
\begin{document}

\IEEEoverridecommandlockouts
\title{\bf Convergence Analysis of Nonconvex Distributed Stochastic Zeroth-order Coordinate Method}

\author{Shengjun~Zhang, Yunlong~Dong, Dong~Xie, Lisha~Yao, Colleen~P.~Bailey, Shengli~Fu
\thanks{Shengjun Zhang, Dong Xie and Colleen P. Bailey are with OSCAR Laboratory, Department of Electrical Engineering, University of North Texas, Denton, TX 76207 USA. {\tt\small \{ShengjunZhang, DongXie\}@my.unt.edu}, {\tt\small Colleen.Bailey@unt.edu}.}
\thanks{Yunlong Dong is with School of Artificial Intelligence and Automation, the MOE Key Laboratory of Image Processing and Intelligent Control, and the State Key Laboratory of Digital Manufacturing Equipment and Technology, Huazhong University of Science and Technology, Wuhan 430074, China. {\tt\small dyl@hust.edu.cn}.}
\thanks{Lisha Yao and Shengli Fu are with the Department of Electrical Engineering, University of North Texas, Denton, TX 76207 USA. {\tt\small Lisha Yao@my.unt.edu}, {\tt\small Shengli.Fu@unt.edu}.}
}

\maketitle

\begin{abstract}            

This paper investigates the stochastic distributed nonconvex optimization problem of minimizing a global cost function formed by the summation of $n$ local cost functions. We solve such a problem by involving zeroth-order (ZO) information exchange.
In this paper, we propose a \underline{ZO} \underline{d}istributed pr\underline{i}mal--du\underline{a}l \underline{c}oordinate method ~(ZODIAC) to solve the stochastic optimization problem. 
Agents approximate their own local stochastic ZO oracle along with coordinates with an adaptive smoothing parameter.
We show that the proposed algorithm achieves the convergence rate of $\mathcal{O}(\sqrt{p}/\sqrt{T})$ for general nonconvex cost functions.
We demonstrate the efficiency of proposed algorithms through a numerical example in comparison with the existing state-of-the-art centralized and distributed ZO algorithms.
\end{abstract}

\section{Introduction\label{sec:Introduction}}

In this paper, we investigate stochastic distributed nonconvex optimization problems with only zeroth-order (ZO) information available.
Such problems can be mathematically summarized in the form:
\begin{align}\label{zerosg:eqn:xopt}
 \min_{x\in \mathbb{R}^p} f(x)=\frac{1}{n}\sum_{i=1}^n\mathbb{E}_{\xi_i}[F_i(x,\xi_i)],
\end{align}
where $n$ is the total number of agents, $x$ is the decision variable, $\xi_i$ is a random variable with dimension $p$, and $F_i(\cdot,\xi_i): \mathbb{R}^{p}\rightarrow \mathbb{R}$ is the stochastic function. 
Agent $i$ is only able to access its own stochastic ZO information $F_i(x,\xi_i)$.
For each agent $i$, the local cost function $f_i(x)$ is the expectation of the ZO information $\mathbb{E}_{\xi_i}[F_i(x,\xi_i)]$.
Agents communicate with their neighbors via an undirected communication network graph $\mathcal G$.

Many algorithms based on first-order gradient information have been proposed in the literature and applied to various applications. 
Unfortunately, in many scenarios, the deceptively simple gradient information is not available or too expensive \cite{conn2009introduction,audet2017derivative,larson2019derivative}.
For instance, in simulation based optimization problems \cite{spall2005introduction}, the gradient information of objective functions is not available.
In the machine learning community, universal attacking of deep neural networks is considered a black-box optimization problem \cite{goodfellow2014explaining, liu2018zeroth, chen2019zo}, where it is too difficult to derive the explicit form of the gradient.
Moreover, in the era of big data, people are dealing with complex data generating processes problems, however, the cost functions of such problems cannot be expressed explicitly\cite{chen2017zoo}.
In addition, decentralized optimization methods in general perform better than centralized ones in terms of robustness, data privacy and computation reduction \cite{nedic2018distributed,Koloskova2019decentralized,yang2019survey}.


Starting from early 1960s, derivative-free optimization (DFO) has been applied in several numerical and statistical problems \cite{hooke1961direct,matyas1965random,nelder1965simplex}.
With the rise of machine learning in the past decades, DFO has gained more attention and been investigated deeply.
Recently, the most popular DFO method is utilizing the ZO information, which is treated as the counterpart of the first-order gradient.
In recent years, distributed optimization problems obtained more and more attention as they can be applied into massive networked systems including 
power systems, sensor networks, smart buildings, and smart
manufacturing \cite{yang2019survey}.
More specifically, \cite{yuan2014randomized,sahu2018distributed,wang2019distributed,
pang2019randomized} focus specifically on distributed ZO gradient descent methods.
Yuan et al proposed distributed ZO with the push-sum technique \cite{yuan2015gradient}, Yu et al extended mirror descent algorithm to distributed settings \cite{yu2019distributed}, and Tang et al provided distributed ZO gradient tracking algorithms \cite{tang2020distributedzero}. Both Hajinezhad et al and Yi et al utilized primal--dual techniques combined with ZO information\cite{hajinezhad2019zone,yi2019linear} and Beznosikov et al considered distributed ZO sliding algorithms \cite{beznosikov2019derivative}. 

Most of the aforementioned algorithms can handle the deterministic form of \eqref{zerosg:eqn:xopt}, e.g. $\min_{x\in \mathbb{R}^p} f(x)=\frac{1}{n}\sum_{i=1}^nF_i(x)$, where $F_i(x)$ is a deterministic function.
For stochastic distributed settings in the exact form of \eqref{zerosg:eqn:xopt}, Hajjinezhad et al are able to solve, however, it requires a very high sampling size of $\mathcal{O}(T)$ to achieve the convergence rate of $\mathcal{O}(p^2n/T)$, which is not practically suitable for high dimensional decision variables \cite{hajinezhad2019zone}.

In this paper, we propose a \underline{ZO} \underline{d}istributed pr\underline{i}mal--du\underline{a}l \underline{c}oordinate method (ZODIAC) 
to solve the stochastic optimization problem~\eqref{zerosg:eqn:xopt}.
To our best knowledge, compared to other existing ZO distributed algorithms, ZODIAC is the only one estimating ZO oracle along with coordinates, which improves the gradient estimation error \cite{pmlr-v97-ji19a}. Compared to \cite{hajinezhad2019zone}, ZODIAC has lower sample requirements and is favorable for large-scale distributed optimization problems in practice.
We show that our algorithm finds a stationary point with a convergence rate of $\mathcal{O}(\sqrt{p}/\sqrt{T})$ for general nonconvex cost functions using a fixed stepsize, which is faster than the centralized ZO algorithms in \cite{ghadimi2013stochastic,lian2016Comprehensive,zhang2020improving,liu2018zerothieee,
liu2019signsgd,balasubramanian2018zeroth,chen2019zo} and the distributed ZO primal algorithm in \cite{tang2020distributedzero}.

The rest of this paper is organized as follows.
Section~\ref{zerosg:sec-preliminary} introduces some preliminary concepts. Sections~\ref{sec:alg} introduces ZODIAC and analyzes its convergence properties. Simulations are presented in Section~\ref{sec:simulations}. Finally, concluding remarks are offered in Section~\ref{sec:conclusions}.

\noindent {\bf Notations}: $\mathbb{N}_0$ and $\mathbb{N}_+$ denote the set of nonnegative and positive integers, respectively. $[n]$ denotes the set $\{1,\dots,n\}$ for any $n\in\mathbb{N}_+$.
$\|\cdot\|$ represents the Euclidean norm for vectors or the induced 2-norm for matrices. $\mathbb{B}^p$ and $\mathbb{S}^p$ are the unit ball and sphere centered around the origin in $\mathbb{R}^p$ under Euclidean norm, respectively. Given a differentiable function $f$, $\nabla f$ denotes the gradient of $f$.
${\bm 1}_n$ (${\bm 0}_n$) denotes the column one (zero) vector of dimension $n$. $\col(z_1,\dots,z_k)$ is the concatenated column vector of vectors $z_i\in\mathbb{R}^{p_i},~i\in[k]$. ${\bm I}_n$ is the $n$-dimensional identity matrix. Given a vector $[x_1,\dots,x_n]^\top\in\mathbb{R}^n$, $\diag([x_1,\dots,x_n])$ is a diagonal matrix with the $i$-th diagonal element being $x_i$.  The notation $A\otimes B$ denotes the Kronecker product
of matrices $A$ and $B$. Moreover, we denote $\bsx=\col(x_1,\dots,x_n)$, $\bar{x}=\frac{1}{n}({\bm 1}_n^\top\otimes{\bm I}_p)\bsx$, $\bar{\bsx}={\bm 1}_n\otimes\bar{x}$.
$\rho(\cdot)$ stands for the spectral radius for matrices and $\rho_2(\cdot)$ indicates the minimum
positive eigenvalue for matrices having positive eigenvalues.


\section{Preliminaries}\label{zerosg:sec-preliminary}
The following section discusses some background on graph theory, smooth functions, the gradient
estimator, and additional assumptions used in this paper.

\subsection{Graph Theory}
Agents communicate with their neighbors through an underlying network, which is modeled by an undirected graph $\mathcal G=(\mathcal V,\mathcal E)$, where $\mathcal V =\{1,\dots,n\}$ is the agent set, $\mathcal E
\subseteq \mathcal V \times \mathcal V$ is the edge set, and $(i,j)\in \mathcal E$ if agents $i$ and $j$ can communicate with each other.
For an undirected graph $\mathcal G=(\mathcal V,\mathcal E)$, let $\mathcal{A}=(a_{ij})$ be the associated weighted adjacency matrix with $a_{ij}>0$ if $(i,j)\in \mathcal E$ if $a_{ij}>0$ and zero otherwise. It is assumed that $a_{ii}=0$ for all $i\in [n]$. Let $\deg_i=\sum\limits_{j=1}^{n}a_{ij}$ denotes the weighted degree of vertex $i$. The degree matrix of graph $\mathcal G$ is $\Deg=\diag([\deg_1, \cdots, \deg_n])$. The Laplacian matrix is $L=(L_{ij})=\Deg-\mathcal{A}$. Additionally, we denote $K_n={\bm I}_n-\frac{1}{n}{\bm 1}_n{\bm 1}^{\top}_n$, $\bsL=L\otimes {\bm I}_p$, $\bsK=K_n\otimes {\bm I}_p$, $\bsH=\frac{1}{n}({\bm 1}_n{\bm 1}_n^\top\otimes{\bm I}_p)$. Moreover, from Lemmas~1 and 2 in \cite{Yi2018distributed}, we know there exists an orthogonal matrix $[r \ R]\in \mathbb{R}^{n \times n}$ with $r=\frac{1}{\sqrt{n}}\mathbf{1}_n$ and $R \in \mathbb{R}^{n\times (n-1)}$ such that $R\Lambda_1^{-1}R^{\top}L=LR\Lambda_1^{-1}R^{\top}=K_n$, and $\frac{1}{\rho(L)}K_n\leq R\Lambda_1^{-1}R^{\top}\le\frac{1}{\rho_2(L)}K_n$, where $\Lambda_1=\diag([\lambda_2,\dots,\lambda_n])$ with $0<\lambda_2\leq\dots\leq\lambda_n$ being the eigenvalues of the Laplacian matrix $L$.

\subsection{Smooth Function}
\begin{definition}
A function $f(x):~\mathbb{R}^p\mapsto\mathbb{R}$ is smooth with constant $L_f>0$ if it is differentiable and
\begin{align}\label{nonconvex:smooth}
\|\nabla f(x)-\nabla f(y)\|\le L_{f}\|x-y\|,~\forall x,y\in \mathbb{R}^p.
\end{align}
\end{definition}

\subsection{Gradient Approximation}

Denote a random subset of the coordinates $\mathcal{S} \subseteq \{1, 2, \dots, p\}$ where the cardinality of $\mathcal{S}$ is $|\mathcal{S}| = n_{c}$. 
We provide two options of gradient approximation, denoted $g^e_{i}$ and defined by~\eqref{dbco:gradient:model2-st} and~\eqref{dbco:gradient:model2-st2}.

\begin{align}
&g^e_{i}=\frac{p}{n_{c}}\sum_{i\in \mathcal{S}}\frac{(F(x+\delta_{i} e_{i},\xi)-F(x,\xi))}{\delta_{i}}e_{i}
\label{dbco:gradient:model2-st}
\end{align}

\begin{align}
&g^e_{i}=\frac{p}{n_{c}}\sum_{i\in \mathcal{S}}\frac{(F(x+\delta_{i} e_{i},\xi)-F(x-\delta_{i} e_{i},\xi))}{2\delta_{i}}e_{i}
\label{dbco:gradient:model2-st2}
\end{align}

The coordinates are sampled uniformly, i.e. $\text{Pr}(i\in \mathcal{S}) = n_{c}/p$, which guarantees that both~\eqref{dbco:gradient:model2-st} and~\eqref{dbco:gradient:model2-st2} are \textit{unbiased} estimators of the \textit{full} coordinate gradient estimator $\sum_{i=1}^{d}\frac{(F(x+\delta_{i} e_{i},\xi)-F(x-\delta_{i} e_{i},\xi))}{2\delta_{i}}e_{i}$\cite{sharma2020zeroth}.

\subsection{Assumptions}

\begin{assumption}\label{zerosg:ass:graph}
The undirected graph $\mathcal G$ is connected.
\end{assumption}

\begin{assumption}\label{zerosg:ass:optset}
The optimal set $\mathbb{X}^*$ is nonempty and the optimal value $f^*>-\infty$.
\end{assumption}

\begin{assumption}\label{zerosg:ass:zeroth-smooth}
For almost all $\xi_i$, the stochastic ZO oracle $F_i(\cdot,\xi_i)$ is smooth with constant $L_f>0$.
\end{assumption}

\begin{assumption}\label{zerosg:ass:zeroth-variance}
The stochastic gradient $\nabla_xF_i(x,\xi_i)$ has bounded variance for any $j$th coordinate of $x$, i.e., there exists $\zeta\in\mathbb{R}$ such that $\mathbb{E}_{\xi_i}[(\nabla_xF_i(x,\xi_i)-\nabla f_i(x))_{j}^2]\le\zeta^2,~\forall i\in[n],~\forall j\in[p],~\forall x\in\mathbb{R}^p$. It also implies that $\mathbb{E}_{\xi_i}[\|\nabla_xF_i(x,\xi_i)-\nabla f_i(x)\|^2]\le\sigma^2_1\triangleq p \zeta^2,~\forall i\in[n],~\forall x\in\mathbb{R}^p$.
\end{assumption}

\begin{assumption}\label{zerosg:ass:fig}
Local cost functions are similar, i.e.,
there exists $\sigma_2\in\mathbb{R}$ such that $\|\nabla f_i(x)-\nabla f(x)\|^2\le\sigma^2_2,~\forall i\in[n],~\forall x\in\mathbb{R}^p$.
\end{assumption}

\begin{remark}
There is no assumption made on convexity. Assumption~\ref{zerosg:ass:graph} and \ref{zerosg:ass:optset} are basic and common in optimization literature.
Assumptions~\ref{zerosg:ass:zeroth-smooth} and \ref{zerosg:ass:zeroth-variance} are standard for solving ZO stochastic optimization problems. 
Assumption~\ref{zerosg:ass:fig} is slightly weaker than stating each $\nabla f_i$ is bounded, which is commonly used in finite-sum type ZO optimization literature.
\end{remark}

\section{Proposed ZODIAC Algorithm}\label{sec:alg}

\subsection{Algorithm Description}

In order to handle stochastic optimization problems, we propose the ZODIAC algorithm, where we consider the novel distributed primal-dual scheme~\cite{yi2020linear} with the stochastic coordinate estimators~\eqref{dbco:gradient:model2-st} and~\eqref{dbco:gradient:model2-st2}, summarized in Algorithm~\ref{zerosc:algorithm-random-pd}.

\begin{subequations}\label{zerosg:alg:random-pd}
\begin{align}
x_{i,k+1} &= x_{i,k}-\eta\Big(\alpha\sum_{j=1}^nL_{ij}x_{j,k}+\beta v_{i,k}+g^e_{i,k}\Big), \label{zerosg:alg:random-pd-x}\\
v_{i,k+1} &=v_{i,k}+ \eta\beta\sum_{j=1}^n L_{ij}x_{j,k},\notag\\
&\quad\forall x_{i,0}\in\mathbb{R}^p, ~\sum_{j=1}^nv_{j,0}={\bm 0}_p,~
\forall i\in[n].  \label{zerosg:alg:random-pd-q}
\end{align}
\end{subequations}

\begin{algorithm}[th!]
\caption{ZODIAC}
\label{zerosc:algorithm-random-pd}
\begin{algorithmic}[1]
\STATE \textbf{Input}: positive number $\alpha$, $\beta$, $\eta$, and positive sequences $\{\delta_{i,k}\}$.
\STATE \textbf{Initialize}: $ x_{i,0}\in\mathbb{R}^p$ and $v_{i,0}={\bm 0}_p,~
\forall i\in[n]$.
\FOR{$k=0,1,\dots$}
\FOR{$i=1,\dots,n$  in parallel}
\STATE  Broadcast $x_{i,k}$ to $\mathcal{N}_i$ and receive $x_{j,k}$ from $j\in\mathcal{N}_i$;
\STATE Select coordinates independently and uniformly;
\STATE Select $\xi_{i,k}$ independently;\\
\STATE \textbf{Option 1:} sample $F_i(x_{i,k}+\delta_{i,k}e_{i,k},\xi_{i,k})$, and $F_i(x_{i,k},\xi_{i,k})$;
\STATE  Update $x_{i,k+1}$ by \eqref{zerosg:alg:random-pd-x} with~\eqref{dbco:gradient:model2-st};
\STATE \textbf{Option 2:} sample  $F_i(x_{i,k}+\delta_{i,k}e_{i,k},\xi_{i,k})$ and $F_i(x_{i,k}-\delta_{i,k}e_{i,k},\xi_{i,k})$;
\STATE  Update $x_{i,k+1}$ by \eqref{zerosg:alg:random-pd-x} with~\eqref{dbco:gradient:model2-st2};
\STATE  Update $v_{i,k+1}$ by \eqref{zerosg:alg:random-pd-q}.
\ENDFOR
\ENDFOR
\STATE  \textbf{Output}: $\{\bsx_{k}\}$.
\end{algorithmic}
\end{algorithm}

Algorithm~\eqref{zerosc:algorithm-random-pd} can be written in compact form as 
\begin{subequations}\label{zerosg:alg:random-pd-compact}
\begin{align}
\bm{x}_{k+1}&=\bm{x}_k-\eta(\alpha\bsL\bm{x}_k+\beta\bm{v}_k+\bsg^e_k),\label{zerosg:alg:random-pd-compact-x}\\
\bm{v}_{k+1}&=\bm{v}_k+\eta\beta\bsL\bm{x}_k,~\forall \bsx_0\in\mathbb{R}^{np},~\sum_{i=1}^{n}v_{i,0}={\bm 0}_p.\label{zerosg:alg:random-pd-compact-v}
\end{align}
\end{subequations}

\subsection{Convergence Analysis}

\begin{theorem}\label{zerosg:thm-sg-smT}
Suppose Assumptions~\ref{zerosg:ass:graph}--\ref{zerosg:ass:fig} hold. For any given $T>n^3/p$, let $\{\bsx_k,k=0,\dots,T\}$ be the output generated by Algorithm~\ref{zerosc:algorithm-random-pd} with
\begin{align}\label{zerosg:step:eta2-sm}
&\alpha=\kappa_1\beta,~\beta=\frac{\kappa_2\sqrt{pT}}{\sqrt{n}},~ \eta=\frac{\kappa_2}{\beta},\nonumber\\
&\delta_{i,k}\le\frac{\kappa_\delta}{p^{\frac{1}{4}}n^{\frac{1}{4}}(k+1)^{\frac{1}{4}}},~\forall k\le T,
\end{align}
where $\kappa_1>\frac{1}{\rho_2(L)}+1$, $\kappa_2\in\Big(0,\min\{\frac{(\kappa_1-1)\rho_2(L)-1}{\rho(L)+(2\kappa_1^2+1)\rho(L^2)+1},\frac{1}{5}\}\Big)$,  and $\kappa_\delta>0$, then we have,
\begin{subequations}
\begin{align}
&\frac{1}{T}\sum_{k=0}^{T-1}\mathbb{E}[\|\nabla f(\bar{x}_k)\|^2]
=\mathcal{O}(\frac{\sqrt{p}}{\sqrt{	T}})
+\mathcal{O}(\frac{n}{T}),\label{zerosg:coro-sg-sm-equ3}\\
&\mathbb{E}[f(\bar{x}_{T})]-f^*=\mathcal{O}(1),\label{zerosg:coro-sg-sm-equ4}\\
&\frac{1}{T}\sum_{k=0}^{T-1}\mathbb{E}\Big[\frac{1}{n}\sum_{i=1}^{n}\|x_{i,k}-\bar{x}_k\|^2\Big]
=\mathcal{O}(\frac{n}{T}).\label{zerosg:coro-sg-sm-equ3.1}
\end{align}
\end{subequations}
\end{theorem}

Before proving Theorem~\ref{zerosg:thm-sg-smT}, we introduce the following lemmas.

\begin{lemma}\label{zerosg:lemma:variance}
Consider $f(x) = \mathbb{E}_{\xi} [F(x, \xi)]$, we have the following relationship,
\begin{align}
&\mathbb{E} \Big[ \| g^e_{i} \|  ^2\Big] \nonumber \\
&\leq 2(p-1)\left\| \nabla f (x) \right\|^2 + 2p \sigma^2_1 + \frac{3p^{2}}{n_{c}} \left( \zeta^2 + \frac{L_{f}^2 \delta_{k}^2}{2} \right) \nonumber \\
&\quad+ \frac{p^{2}L_{f}^2 \delta_{k}^2}{2} \label{eq_bd_var_CGE}
\end{align}
\end{lemma}
where $\delta_{k} = \max\{ \delta_{i} \}, i \in [p]$.
\begin{proof}
Apply the proposition III.2 in \cite{sharma2020zeroth} and consider the coordinates are picked uniformly, then we have
\begin{align}
&\mathbb{E} \Big[ \| g^e_{i} - \nabla f (x) \|  ^2\Big] \nonumber \\
&\leq \sum_{i=1}^p p \Big[ 2 \left( \nabla f (x) \right)_i^2 + \frac{3}{n_{c}} \left( \zeta^2 + \frac{L^2 \delta_{i}^2}{2} \right) + \frac{L^2 \delta_{i}^2}{2} \Big] \nonumber \\
&\quad- 2 \left\| \nabla f (x) \right\|^2 
\end{align}
We can easily get Eq.~\eqref{eq_bd_var_CGE} by simpliying the above inequality.
\end{proof}

\begin{lemma}\label{zerosg:lemma:grad-st}
Suppose Assumptions~\ref{zerosg:ass:zeroth-smooth}--\ref{zerosg:ass:fig} hold. Let $\{\bsx_k\}$ be the sequence generated by Algorithm~\ref{zerosc:algorithm-random-pd}, $\bsg^e_k=\col(g^e_{1,k},\dots,g^e_{n,k})$, $\bsg^0_k=n\nabla{f}(\bar{\bsx}_k)$, $\bar{\bsg}_k^0=\bsH\bsg^0_{k}={\bm 1}_n\otimes\nabla f(\bar{x}_k)$, then
\begin{subequations}
\begin{align}
\mathbb{E}\Big[\|\bsg^e_k\|^2\Big]
&\le  6(p-1)\|\bar{\bsg}_{k}^0\|^2+6(p-1)L_f^2\|\bsx_{k}\|^2_{\bsK}\nonumber\\
&\quad+6n(p-1)\sigma^2_2 + \frac{3np^{2}}{n_{c}} \left( \zeta^2 + \frac{L_{f}^2 \delta_{k}^2}{2} \right)\nonumber \\
&\quad+2np \sigma^2_1+ \frac{np^{2}L_{f}^2 \delta_{k}^2}{2}\label{zerosg:rand-grad-esti2}\\
\|\bsg^0_{k+1}\|^2&\le 3(\eta^2L_f^2\|\bsg^e_{k}\|^2+n\sigma^2_2
+\|\bar{\bsg}_{k}^0\|^2).\label{zerosg:rand-grad-esti4}
\end{align}
\end{subequations}
\end{lemma}
\begin{proof}
(i) Eq.~\eqref{zerosg:rand-grad-esti2} is due to Lemma~\ref{zerosg:lemma:variance}, Cauchy-Schwarz inequality and Assumption~\ref{zerosg:ass:fig}.

(ii) Eq.~\eqref{zerosg:rand-grad-esti4} is established by Cauchy-Schwarz inequality, Assumption~\ref{zerosg:ass:zeroth-smooth} and~\ref{zerosg:ass:fig}.
\end{proof}

\begin{lemma}\label{zerosg:lemma:sg2-T}
Suppose Assumptions~\ref{zerosg:ass:graph}--\ref{zerosg:ass:fig} hold, and we have fixed parameters $\alpha=\kappa_1\beta$, $\beta$, and $\eta=\frac{\kappa_2}{\beta}$, where $\beta$ is large enough,
$\kappa_1>\frac{1}{\rho_2(L)}+1$ and $\kappa_2\in\Big(0,\min\{\frac{(\kappa_1-1)\rho_2(L)-1}{\rho(L)+(2\kappa_1^2+1)\rho(L^2)+1}, \frac{1}{5}\}\Big)$ are constants. Let $\{\bsx_k\}$ be the sequence generated by Algorithm~\ref{zerosc:algorithm-random-pd}, then
\begin{subequations}
\begin{align}
\mathbb{E}[W_{k+1}]
&\le   W_{k}-\kappa_4\|\bsx_k\|^2_{\bsK}\nonumber\\
&\quad-\frac{1}{2}(\kappa_2 - 5\kappa_2^2)\Big\|\bm{v}_k+\frac{1}{\beta}\bsg_{k}^0\Big\|^2_{\bsK}\nonumber\\
&\quad-\frac{1}{8}\eta\|\bar{\bsg}^0_{k}\|^2+\mathcal{O}(np)\eta^2 + \mathcal{O}(np^2)\eta\delta_k^2,
\label{zerosg:sgproof-vkLya2T}\\
\mathbb{E}[W_{4,k+1}]
&\le  W_{4,k}+2\eta L_f^2\|\bsx_k\|^2_{\bsK}-\frac{1}{8}\eta\|\bar{\bsg}_{k}^0\|^2\nonumber\\
&\quad+\mathcal{O}(p)\eta^2
+\mathcal{O}(np)\eta\delta^2_k.\label{zerosg:v4kspeed}
\end{align}
\end{subequations}
\end{lemma}

\begin{proof}
We provide the proof of Lemma~\ref{zerosg:lemma:sg2-T} in the appendix~\ref{proof:lemma1}.
\end{proof}

We are now ready to prove Theorem~\ref{zerosg:thm-sg-smT}.
\begin{proof}

Denote
\begin{align*}
\hat{V}_k=\|\bm{x}_k\|^2_{\bsK}+\Big\|\bsv_k
+\frac{1}{\beta_k}\bsg_k^0\Big\|^2_{\bsK}+n(f(\bar{x}_k)-f^*).
\end{align*}
We have
\begin{align}
&W_{k}\nonumber\\
&=\frac{1}{2}\|\bsx_{k}\|^2_{\bsK}
+\frac{1}{2}\Big\|\bsv_k+\frac{1}{\beta_k}\bsg_k^0\Big\|^2_{\bsQ+\kappa_1\bsK}\nonumber\\
&~~~+\bsx_k^\top\bsK\Big(\bm{v}_k+\frac{1}{\beta_k}\bsg_k^0\Big)+n(f(\bar{x}_k)-f^*)\nonumber\\
&\ge\frac{1}{2}\|\bsx_{k}\|^2_{\bsK}
+\frac{1}{2}\Big(\frac{1}{\rho(L)}+\kappa_1\Big)
\Big\|\bsv_k+\frac{1}{\beta_k}\bsg_k^0\Big\|^2_{\bsK}\nonumber\\
&~~~-\frac{1}{2\kappa_1}\|\bsx_{k}\|^2_{\bsK}
-\frac{1}{2}\kappa_1\Big\|\bsv_k+\frac{1}{\beta_k}\bsg_k^0\Big\|^2_{\bsK}
+n(f(\bar{x}_k)-f^*)\nonumber\\
&\ge\min\Big\{\frac{1}{2\rho(L)},~\frac{\kappa_1-1}{2\kappa_1}\Big\}\hat{V}_k\ge0,\label{zerosg:vkLya3}
\end{align}
Additionally, we can get $W_{k}\le(\frac{\kappa_1+1}{2}+\frac{1}{2\rho_2(L)})\hat{V}_k$.
	
Consider that $\beta=\kappa_2\sqrt{pT}/\sqrt{n}$ and $T> n^3/p$, we know that Lemma~\ref{zerosg:lemma:sg2-T} are satisfied. So \eqref{zerosg:sgproof-vkLya2T} and \eqref{zerosg:v4kspeed} hold.
Summing~\eqref{zerosg:sgproof-vkLya2T} over $k \in [0, T]$ and applying~\eqref{zerosg:vkLya3}, we have
\begin{align}
&\frac{1}{T+1}\sum_{k=0}^{T}\mathbb{E}[\frac{1}{n}\sum_{i=1}^{n}\|x_{i,k}-\bar{x}_k\|^2]\nonumber\\
&\le\frac{1}{\kappa_4}\Big(\frac{W_{0}}{n(T+1)}
+\frac{\mathcal{O}(n)\eta\delta_k^2}{T}
+\frac{\mathcal{O}(n/p)\eta^2\kappa_\delta}{\sqrt{T(T+1)}}\Big) \nonumber\\
&=\mathcal{O}(\frac{n}{T}),
\label{zerosg:thm-sg-sm-equ3.1p}
\end{align}
where $W_{0} = \mathcal{O}(n)$, $\frac{W_{0}}{n(T+1)}=\mathcal{O}(\frac{1}{T})$, $\frac{n\mathcal{O}(p^2)\eta\delta_k^2}{T} = \mathcal{O}(\frac{n}{T})$, and $\frac{\mathcal{O}(n/p)\eta^2\kappa_\delta}{\sqrt{T(T+1)}}=\mathcal{O}(\frac{n}{pT})$ , which gives \eqref{zerosg:coro-sg-sm-equ3.1}.

From~\eqref{zerosg:v4kspeed},~\eqref{zerosg:step:eta2-sm}, and~\eqref{zerosg:vkLya3}, summing~\eqref{zerosg:v4kspeed} over $k \in [0, T]$ similar to the way to get \eqref{zerosg:coro-sg-sm-equ3.1}, we have
\begin{align}\label{zerosg:thm-sg-sm-equ3p}
&\frac{1}{T+1}\sum_{k=0}^{T}\mathbb{E}[\|\nabla f(\bar{x}_k)\|^2]=\frac{1}{n(T+1)}\sum_{k=0}^{T}\mathbb{E}[\|\bar{\bsg}_{k}^0\|^2]\nonumber\\
&\le 8\Big(\frac{W_{4,0}}{n(T+1)\eta}
+\frac{2L_f^2}{n(T+1)}\sum_{k=0}^{T}\mathbb{E}[\|\bsx_k\|^2_{\bsK}]+\frac{\mathcal{O}(p)}{n}\nonumber\\
&~~~+\frac{\mathcal{O}(\sqrt{np})}{\sqrt{n(T+1)}}\Big).
\end{align}
Noting that $\eta=\kappa_2/\beta_k=\sqrt{n}/\sqrt{pT}$, and $n/T<\sqrt{p}/\sqrt{nT}$ due to $T> n^3/p$, from \eqref{zerosg:thm-sg-sm-equ3p} and \eqref{zerosg:thm-sg-sm-equ3.1p}, we have
\begin{align*}
\frac{1}{T}\sum_{k=0}^{T-1}\mathbb{E}[\|\nabla f(\bar{x}_k)\|^2]
&=\mathcal{O}(\frac{\sqrt{p}}{\sqrt{T}})
+\mathcal{O}(\frac{n}{T}),
\end{align*}
which gives \eqref{zerosg:coro-sg-sm-equ3}.

Summing \eqref{zerosg:v4kspeed} over $ k\in[0,T]$, and using \eqref{zerosg:step:eta2-sm}  yield
\begin{align}\label{zerosg:thm-sg-sm-equ4p}
&n(\mathbb{E}[f(\bar{x}_{T+1})]-f^*)=\mathbb{E}[W_{4,T+1}]\nonumber\\
&\le W_{4,0}+\frac{2\sqrt{n}}{\sqrt{pT}} L_f^2\sum_{k=0}^{T}\|\bsx_k\|^2_{\bsK}+n\mathcal{O}(p)\eta^2\frac{T+1}{T}\nonumber\\
&~~~+\mathcal{O}(np)\eta\delta^2_k\sqrt{\frac{T+1}{T}}.
\end{align}

Noting that $W_{4,0}=\mathcal{O}(n)$ and $\sqrt{n}n/\sqrt{pT}<1$ due to $T> n^3/p$, from \eqref{zerosg:thm-sg-sm-equ3.1p} and \eqref{zerosg:thm-sg-sm-equ4p}, we have $\mathbb{E}[f(\bar{x}_{T+1})]-f^*=\mathcal{O}(1)$, which gives~\eqref{zerosg:coro-sg-sm-equ4}.

\end{proof}

\section{Numerical Examples}\label{sec:simulations}
We consider a benchmark non-linear least square problem from the literature \cite{liu2019signsgd, liu2018zeroth}. 
The local cost function is given as $f_i(\mathbf x) = \left ( y_i - \phi(\mathbf x; \mathbf a_i) \right )^2$ for $i\in[n]$, where $\phi(\mathbf x; \mathbf a_i) = \frac{1}{1 + e^{-\mathbf a_i^{T} \mathbf x}}$, $\xi_i$ follows a standard normal distribution $\mathcal{N}_{i}(0, 0.01)$.
To prepare the synthetic dataset, we randomly draw samples $\mathbf a_i$ from $\mathcal{N}(\mathbf 0, \mathbf{I} )$ and set an optimal vector $\mathbf{x_{opt}} = \mathbf{1}$. The label is $y_i = 1$ if $\phi(\mathbf x_{opt}; \mathbf a_i) \geq 0.5$ and $0$ otherwise. The training set has $2000$ samples and the test set has $200$ samples. 
We set the dimension $d$ of $\mathbf a_i$ to $100$, the batch size is $1$, and the total iteration number is $50000$. As suggested in the work \cite{liu2019signsgd}, the smooth parameter $\delta = \frac{10}{\sqrt{Td}}$.
The communication topology of $10$ agents is generated randomly following the Erd\H{o}s - R\' enyi model with the connection probability of $0.4$. 

We compare the proposed ZODIAC algorithm with the two estimator options,~\eqref{dbco:gradient:model2-st} and~\eqref{dbco:gradient:model2-st2}, against the current state-of-the-art centralized and distributed ZO algorithms: ZO-SGD \cite{ghadimi2013stochastic}, ZO-SCD \cite{lian2016Comprehensive}, distributed ZO gradient tracking algorithm (ZO-GDA) \cite{tang2020distributedzero} and ZONE-M \cite{hajinezhad2019zone}. The hyper-parameters used in the experiments are well-tuned based on performance and provided in Table~\ref{tab:para-bc}. The test accuracy of each algorithm is summarized in Table~\ref{tab:acc}. 
From Fig.~\ref{fig:loss}, we can see that ZODIAC outperforms the existing algorithms and achieve better loss results. Additionally, both ZODIAC implementations have higher accuracy.
Moreover, we provide the error of the gradient estimation in ZODICA in Fig.~\ref{fig:grad_est}.

\begin{figure}[ht!]
  \centering
  \includegraphics[width=1\linewidth]{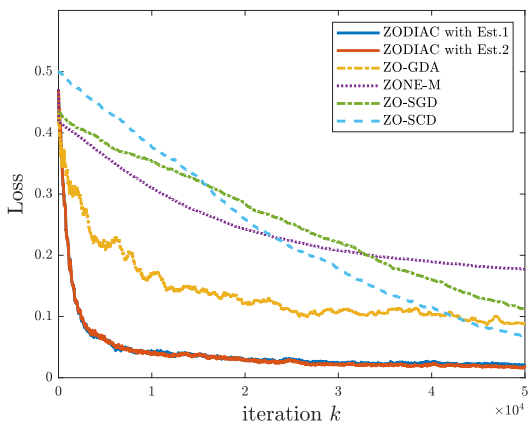}  
  \caption{Training loss evaluations.}
  \label{fig:loss}
\end{figure}

\begin{figure}[ht!]
  \centering
  \includegraphics[width=1\linewidth]{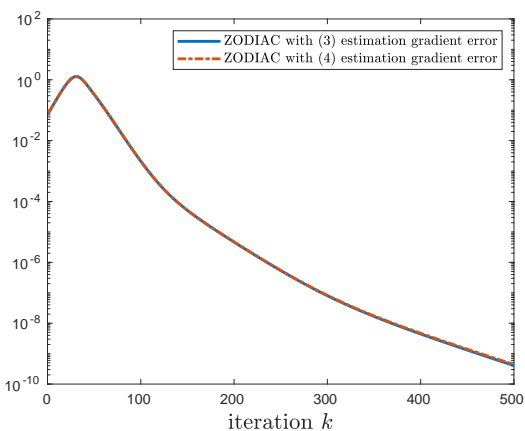}  
  \caption{The error of the gradient estimation.}
  \label{fig:grad_est}
\end{figure}

\begin{table}[ht!]
\caption{Parameters for Binary Classification}
\label{tab:para-bc}
\begin{center}\scalebox{1}{
\begin{tabular}{cc}
\multicolumn{1}{c}{Algorithm}  &\multicolumn{1}{c}{Parameters}
\\ \hline
ZODIAC \footnotemark  & $\eta = 0.08$, $\alpha = 4$, $\beta = 3$ \\
ZO-SGD  \cite{ghadimi2013stochastic}          & $\mu = 0.01$ \\
ZO-SCD  \cite{lian2016Comprehensive}          & $\mu = 0.01$ \\
ZO-GDA  \cite{tang2020distributedzero}          & $\eta = 0.08/{k^{10^{-5}}}$ \\
ZONE-M \cite{hajinezhad2019zone}        & $\rho = 0.1\sqrt{k}$ \\

\hline
\end{tabular}
}
\end{center}
\end{table}

\footnotetext{ZODIAC is tested under the same parameters for both estimators~\eqref{dbco:gradient:model2-st} and~\eqref{dbco:gradient:model2-st2}.}

\begin{table}[ht!]
\caption{Accuracy}
\label{tab:acc}
\begin{center}
\begin{tabular}{cc}
\multicolumn{1}{c}{Algorithm}  & Accuracy($\%$)
\\ \hline
ZODIAC with~\eqref{dbco:gradient:model2-st}     & \textbf{99.0}  \\
ZODIAC with~\eqref{dbco:gradient:model2-st2}       & \textbf{98.5}  \\
ZO-SGD  \cite{ghadimi2013stochastic}        & 85.5  \\
ZO-SCD  \cite{lian2016Comprehensive}          & 91.0 \\
ZO-GDA  \cite{tang2020distributedzero}       & 91.0  \\
ZONE-M \cite{hajinezhad2019zone}        & 89.5 \\
\hline
\end{tabular}
\end{center}
\end{table}

\section{Conclusions}\label{sec:conclusions}

In this paper, we investigated the stochastic distributed nonconvex optimization problem and proposed a stochastic coordinate method within a primal--dual scheme, ZODIAC.
We demonstrated that the proposed algorithm achieves the convergence rate of $\mathcal{O}(\sqrt{p}/\sqrt{T})$ for general nonconvex cost functions.
Additionally, we illustrated the efficacy and accuracy of ZODIAC through a benchmark example in comparison with the existing state-of-the-art centralized and distributed ZO algorithms.

\section*{Acknowledgments}
The authors would like to thank Dr. Xinlei Yi for his insightful inspirations and motivations on this work.

\bibliographystyle{IEEEtran}
\bibliography{zeroth_order}

\begin{thebibliography}{10}
\providecommand{\url}[1]{#1}
\csname url@samestyle\endcsname
\providecommand{\newblock}{\relax}
\providecommand{\bibinfo}[2]{#2}
\providecommand{\BIBentrySTDinterwordspacing}{\spaceskip=0pt\relax}
\providecommand{\BIBentryALTinterwordstretchfactor}{4}
\providecommand{\BIBentryALTinterwordspacing}{\spaceskip=\fontdimen2\font plus
\BIBentryALTinterwordstretchfactor\fontdimen3\font minus
  \fontdimen4\font\relax}
\providecommand{\BIBforeignlanguage}[2]{{%
\expandafter\ifx\csname l@#1\endcsname\relax
\typeout{** WARNING: IEEEtran.bst: No hyphenation pattern has been}%
\typeout{** loaded for the language `#1'. Using the pattern for}%
\typeout{** the default language instead.}%
\else
\language=\csname l@#1\endcsname
\fi
#2}}
\providecommand{\BIBdecl}{\relax}
\BIBdecl

\bibitem{conn2009introduction}
A.~R. Conn, K.~Scheinberg, and L.~N. Vicente, \emph{Introduction to
  Derivative-Free Optimization}.\hskip 1em plus 0.5em minus 0.4em\relax
  MPS-SIAM Series on Optimization. SIAM Philadelphia, 2009.

\bibitem{audet2017derivative}
C.~Audet and W.~Hare, \emph{Derivative-Free and Blackbox Optimization}.\hskip
  1em plus 0.5em minus 0.4em\relax Springer, 2017.

\bibitem{larson2019derivative}
J.~Larson, M.~Menickelly, and S.~M. Wild, ``Derivative-free optimization
  methods,'' \emph{Acta Numerica}, vol.~28, pp. 287--404, 2019.

\bibitem{spall2005introduction}
J.~C. Spall, \emph{Introduction to stochastic search and optimization:
  estimation, simulation, and control}.\hskip 1em plus 0.5em minus 0.4em\relax
  John Wiley \& Sons, 2005, vol.~65.

\bibitem{goodfellow2014explaining}
I.~J. Goodfellow, J.~Shlens, and C.~Szegedy, ``Explaining and harnessing
  adversarial examples,'' \emph{arXiv preprint arXiv:1412.6572}, 2014.

\bibitem{liu2018zeroth}
S.~Liu, B.~Kailkhura, P.-Y. Chen, P.~Ting, S.~Chang, and L.~Amini,
  ``Zeroth-order stochastic variance reduction for nonconvex optimization,'' in
  \emph{Advances in Neural Information Processing Systems}, 2018, pp.
  3727--3737.

\bibitem{chen2019zo}
X.~Chen, S.~Liu, K.~Xu, X.~Li, X.~Lin, M.~Hong, and D.~Cox, ``{ZO}-{A}da{MM}:
  Zeroth-order adaptive momentum method for black-box optimization,'' in
  \emph{Advances in Neural Information Processing Systems}, 2019, pp.
  7204--7215.

\bibitem{chen2017zoo}
P.-Y. Chen, H.~Zhang, Y.~Sharma, J.~Yi, and C.-J. Hsieh, ``{ZOO}: Zeroth order
  optimization based black-box attacks to deep neural networks without training
  substitute models,'' in \emph{ACM Workshop on Artificial Intelligence and
  Security}, 2017, pp. 15--26.

\bibitem{nedic2018distributed}
A.~Nedi{\'c} and J.~Liu, ``Distributed optimization for control,'' \emph{Annual
  Review of Control, Robotics, and Autonomous Systems}, vol.~1, pp. 77--103,
  2018.

\bibitem{Koloskova2019decentralized}
A.~Koloskova, S.~Stich, and M.~Jaggi, ``Decentralized stochastic optimization
  and gossip algorithms with compressed communication,'' in \emph{International
  Conference on Machine Learning}, 2019, pp. 3478--3487.

\bibitem{yang2019survey}
T.~Yang, X.~Yi, J.~Wu, Y.~Yuan, D.~Wu, Z.~Meng, Y.~Hong, H.~Wang, Z.~Lin, and
  K.~H. Johansson, ``A survey of distributed optimization,'' \emph{Annual
  Reviews in Control}, vol.~47, pp. 278--305, 2019.

\bibitem{hooke1961direct}
R.~Hooke and T.~A. Jeeves, ````direct search'' solution of numerical and
  statistical problems,'' \emph{Journal of the ACM}, vol.~8, no.~2, pp.
  212--229, 1961.

\bibitem{matyas1965random}
J.~Matyas, ``Random optimization,'' \emph{Automation and Remote Control},
  vol.~26, no.~2, pp. 246--253, 1965.

\bibitem{nelder1965simplex}
J.~A. Nelder and R.~Mead, ``A simplex method for function minimization,''
  \emph{The Computer Journal}, vol.~7, no.~4, pp. 308--313, 1965.

\bibitem{yuan2014randomized}
D.~Yuan and D.~W. Ho, ``Randomized gradient-free method for multiagent
  optimization over time-varying networks,'' \emph{IEEE Transactions on Neural
  Networks and Learning Systems}, vol.~26, no.~6, pp. 1342--1347, 2014.

\bibitem{sahu2018distributed}
A.~K. Sahu, D.~Jakovetic, D.~Bajovic, and S.~Kar, ``Distributed zeroth order
  optimization over random networks: A {K}iefer-{W}olfowitz stochastic
  approximation approach,'' in \emph{IEEE Conference on Decision and Control},
  2018, pp. 4951--4958.

\bibitem{wang2019distributed}
Y.~Wang, W.~Zhao, Y.~Hong, and M.~Zamani, ``Distributed subgradient-free
  stochastic optimization algorithm for nonsmooth convex functions over
  time-varying networks,'' \emph{SIAM Journal on Control and Optimization},
  vol.~57, no.~4, pp. 2821--2842, 2019.

\bibitem{pang2019randomized}
Y.~Pang and G.~Hu, ``Randomized gradient-free distributed optimization methods
  for a multi-agent system with unknown cost function,'' \emph{IEEE
  Transactions on Automatic Control}, vol.~65, no.~1, pp. 333--340, 2020.

\bibitem{yuan2015gradient}
D.~Yuan, S.~Xu, and J.~Lu, ``Gradient-free method for distributed multi-agent
  optimization via push-sum algorithms,'' \emph{International Journal of Robust
  and Nonlinear Control}, vol.~25, no.~10, pp. 1569--1580, 2015.

\bibitem{yu2019distributed}
Z.~Yu, D.~W. Ho, and D.~Yuan, ``Distributed randomized gradient-free mirror
  descent algorithm for constrained optimization,'' \emph{arXiv preprint
  arXiv:1903.04157}, 2019.

\bibitem{tang2020distributedzero}
Y.~Tang, J.~Zhang, and N.~Li, ``Distributed zero-order algorithms for nonconvex
  multi-agent optimization,'' \emph{arXiv preprint arXiv:1908.11444v3}, 2020.

\bibitem{hajinezhad2019zone}
D.~Hajinezhad, M.~Hong, and A.~Garcia, ``{ZONE}: Zeroth-order nonconvex
  multiagent optimization over networks,'' \emph{IEEE Transactions on Automatic
  Control}, vol.~64, no.~10, pp. 3995--4010, 2019.

\bibitem{yi2019linear}
X.~Yi, S.~Zhang, T.~Yang, T.~Chai, and K.~H. Johansson, ``Linear convergence of
  first- and zeroth-order algorithms for distributed nonconvex optimization
  under the {P}olyak-{{\L}}ojasiewicz condition,'' \emph{arXiv preprint
  arXiv:1912.12110}, 2019.

\bibitem{beznosikov2019derivative}
A.~Beznosikov, E.~Gorbunov, and A.~Gasnikov, ``Derivative-free method for
  composite optimization with applications to decentralized distributed
  optimization,'' \emph{arXiv preprint arXiv:1911.10645v4}, 2020.

\bibitem{pmlr-v97-ji19a}
K.~Ji, Z.~Wang, Y.~Zhou, and Y.~Liang, ``Improved zeroth-order variance reduced
  algorithms and analysis for nonconvex optimization,'' in \emph{International
  Conference on Machine Learning}, 2019, pp. 3100--3109.

\bibitem{ghadimi2013stochastic}
S.~Ghadimi and G.~Lan, ``Stochastic first-and zeroth-order methods for
  nonconvex stochastic programming,'' \emph{SIAM Journal on Optimization},
  vol.~23, no.~4, pp. 2341--2368, 2013.

\bibitem{lian2016Comprehensive}
X.~Lian, H.~Zhang, C.-J. Hsieh, Y.~Huang, and J.~Liu, ``A comprehensive linear
  speedup analysis for asynchronous stochastic parallel optimization from
  zeroth-order to first-order,'' in \emph{Advances in Neural Information
  Processing Systems}, 2016, pp. 3054--3062.

\bibitem{zhang2020improving}
Y.~Zhang, Y.~Zhou, K.~Ji, and M.~M. Zavlanos, ``Improving the convergence rate
  of one-point zeroth-order optimization using residual feedback,'' \emph{arXiv
  preprint arXiv:2006.10820}, 2020.

\bibitem{liu2018zerothieee}
S.~Liu, X.~Li, P.-Y. Chen, J.~Haupt, and L.~Amini, ``Zeroth-order stochastic
  projected gradient descent for nonconvex optimization,'' in \emph{IEEE Global
  Conference on Signal and Information Processing}, 2018, pp. 1179--1183.

\bibitem{liu2019signsgd}
S.~Liu, P.-Y. Chen, X.~Chen, and M.~Hong, ``sign{SGD} via zeroth-order
  oracle,'' in \emph{International Conference on Learning Representations},
  2019.

\bibitem{balasubramanian2018zeroth}
K.~Balasubramanian and S.~Ghadimi, ``Zeroth-order (non)-convex stochastic
  optimization via conditional gradient and gradient updates,'' in
  \emph{Advances in Neural Information Processing Systems}, 2018, pp.
  3455--3464.

\bibitem{Yi2018distributed}
X.~Yi, L.~Yao, T.~Yang, J.~George, and K.~H. Johansson, ``Distributed
  optimization for second-order multi-agent systems with dynamic
  event-triggered communication,'' in \emph{IEEE Conference on Decision and
  Control}, 2018, pp. 3397--3402.

\bibitem{sharma2020zeroth}
P.~Sharma, K.~Xu, S.~Liu, P.-Y. Chen, X.~Lin, and P.~K. Varshney,
  ``Zeroth-order hybrid gradient descent: Towards a principled black-box
  optimization framework,'' \emph{arXiv preprint arXiv:2012.11518}, 2020.

\bibitem{yi2020linear}
X.~Yi, S.~Zhang, T.~Yang, T.~Chai, and K.~H. Johansson, ``Linear convergence
  for distributed optimization without strong convexity,'' in \emph{2020 59th
  IEEE Conference on Decision and Control (CDC)}.\hskip 1em plus 0.5em minus
  0.4em\relax IEEE, 2020, pp. 3643--3648.

\end{thebibliography}
\appendix \label{sec:app}

\section*{Proof of Lemma~\ref{zerosg:lemma:sg2-T}}\label{proof:lemma1}
\begin{proof}

Consider the following Lyapunov candidate function 
\begin{align}
W_{k}= &\underbrace{\frac{1}{2}\|\bsx_{k}\|^2_{\bsK}}_{W_{1, k}} + \underbrace{\frac{1}{2}\Big\|\bsv_k+\frac{1}{\beta}\bsg_k^0\Big\|^2_{\bsQ+\kappa_1\bsK}}_{W_{2, k}} \nonumber \\ &+ \underbrace{\bsx_k^\top\bsK\Big(\bm{v}_k+\frac{1}{\beta}\bsg_k^0\Big)}_{W_{3, k}} + \underbrace{n(f(\bar{x}_k)-f^*)}_{W_{4, k}}
\end{align}
where $\bsQ=R\Lambda^{-1}_1R^{\top}\otimes {\bm I}_p$.
Additionally, we denote $g^s_{i,k}=\nabla \mathbb{E}[f_i(x+\delta_{i,k} e_{i})]$, $\bsg^s_k=\col(g^s_{1,k},\dots,g^s_{n,k})$, $\bar{\bsg}^s_k=\bsH\bsg^s_k$, $\bar{g}^e_k=\frac{1}{n}({\bm 1}_n^\top\otimes{\bm I}_p)\bsg^e_k$, and $\bar{\bsg}^e_k={\bm 1}_n\otimes\bar{g}^e_k=\bsH\bsg^e_k$.

(i) We have
\begin{align}
&\mathbb{E}[W_{1,k+1}]
=\mathbb{E}\Big[\frac{1}{2}\|\bm{x}_{k+1} \|^2_{\bsK}\Big]\nonumber\\
&\overset{\mathrm{Eq.~\ref{zerosg:alg:random-pd-compact-x}}}{=\joinrel=}\mathbb{E}\Big[\frac{1}{2}\|\bm{x}_k-\eta(\alpha\bsL\bm{x}_k+\beta\bm{v}_k+\bsg^e_k) \|^2_{\bsK}\Big]\nonumber\\
&\overset{\text{(a)}}{=}\mathbb{E}\Big[\frac{1}{2}\|\bm{x}_k\|^2_{\bsK}-\eta\alpha\|\bsx_k\|^2_{\bsL}
+\frac{1}{2}\eta^2\alpha^2\|\bsx_k\|^2_{\bsL^2}
\nonumber\\
&~~~-\eta\beta\bsx^\top_k({\bm I}_{np}-\eta\alpha\bsL)\bsK\Big(\bm{v}_k+\frac{1}{\beta}\bsg^e_k\Big)\nonumber\\
&~~~+\frac{1}{2}\eta^2\beta^2\Big\|\bm{v}_k+\frac{1}{\beta}\bsg^e_k\Big\|^2_{\bsK}\Big]\nonumber\\
&\overset{\text{(b)}}{=}W_{1,k}-\|\bsx_k\|^2_{\eta\alpha\bsL
-\frac{1}{2}\eta^2\alpha^2\bsL^2}\nonumber\\
&~~~-\eta\beta\bsx^\top_k({\bm I}_{np}-\eta\alpha\bsL)\bsK\Big(\bm{v}_k
+\frac{1}{\beta}\bsg^s_k\Big)\nonumber\\
&~~~+\frac{1}{2}\eta^2\beta^2\mathbb{E}\Big[\Big\|\bm{v}_k+\frac{1}{\beta}\bsg_k^0
+\frac{1}{\beta}\bsg^e_k-\frac{1}{\beta}\bsg_k^0\Big\|^2_{\bsK}\Big]\nonumber\\
&\overset{\text{(c)}}{\le} W_{1,k}-\|\bsx_k\|^2_{\eta\alpha\bsL
-\frac{1}{2}\eta^2\alpha^2\bsL^2}\nonumber\\
&~~~-\eta\beta\bsx^\top_k\bsK\Big(\bm{v}_k+\frac{1}{\beta}\bsg_k^0\Big)\nonumber\\
&~~~+\frac{1}{2}\eta\|\bm{x}_k\|^2_{\bsK}
+\frac{1}{2}\eta\|\bsg^s_k-\bsg_k^0\|^2\nonumber\\
&~~~+\frac{1}{2}\eta^2\alpha^2\|\bm{x}_k\|^2_{\bsL^2}
+\frac{1}{2}\eta^2\beta^2\Big\|\bm{v}_k+\frac{1}{\beta}\bsg_k^0\Big\|^2_{\bsK}\nonumber\\
&~~~+\frac{1}{2}\eta^2\alpha^2\|\bm{x}_k\|^2_{\bsL^2}
+\frac{1}{2}\eta^2\|\bsg^s_k-\bsg_k^0\|^2\nonumber\\
&~~~+\eta^2\beta^2\Big\|\bm{v}_k+\frac{1}{\beta}\bsg_k^0\Big\|^2_{\bsK}
+\eta^2\mathbb{E}[\|\bsg^e_k-\bsg_k^0\|^2]\nonumber\\
&\overset{\text{(d)}}{\le} W_{1,k}-\|\bsx_k\|^2_{\eta\alpha\bsL-\frac{1}{2}\eta\bsK
-\frac{3}{2}\eta^2\alpha^2\bsL^2-\eta(1+5\eta)L_f^2\bsK}
\nonumber\\
&~~~-\eta\beta\bsx^\top_k\bsK\Big(\bm{v}_k+\frac{1}{\beta}\bsg_k^0\Big)
+\Big\|\bm{v}_k+\frac{1}{\beta}\bsg_k^0\Big\|^2_{\frac{3}{2}\eta^2\beta^2\bsK}\nonumber\\
&~~~+nL_f^2\eta\Big[\frac{p}{4}+(\frac{p}{4}+4)\eta\Big]\delta_k^2+2\eta^2\mathbb{E}[\|\bsg^e_k\|^2],\label{zerosg:v1k}
\end{align}
where (a) holds due to Lemma~1 and~2 in~\cite{Yi2018distributed}; (b) holds due to $\mathbb{E}[\bsg^e_k] = \bsg^s_k$ and that $x_{i,k}$ and $v_{i,k}$ are independent of $u_{i,k}$ and $\xi_{i,k}$; (c) holds due to the Cauchy--Schwarz inequality and $\rho(\bsK)=1$; (d) holds due to $\|\bsg^s_k-\bsg_k^0\|^2\le 2L_f^2\|\bsx_{k}\|^2_{\bsK}+\frac{np}{2}L_f^2\delta_k^2$ and $\mathbb{E}[\|\bsg_k^0-\bsg^e_k\|^2]\le 4L_f^2\|\bsx_{k}\|^2_{\bsK}+4nL_f^2\delta_k^2+2\mathbb{E}[\|\bsg^e_k\|^2]$.

(ii)
\begin{align}
W_{2,k+1}
=\frac{1}{2}\Big\|\bsv_{k+1}+\frac{1}{\beta}\bsg_{k+1}^0\Big\|^2_{\bsQ+\kappa_1\bsK}\label{zerosg:v2k-1}
\end{align}
\begin{align}
&\overset{\mathrm{Eq.~\ref{zerosg:alg:random-pd-compact-v}}}{=\joinrel=}\frac{1}{2}\Big\|\bm{v}_k+\frac{1}{\beta}\bsg_{k}^0+\eta\beta\bsL\bm{x}_k
+\frac{1}{\beta}(\bsg_{k+1}^0-\bsg_{k}^0) \Big\|^2_{\bsQ+\kappa_1\bsK}\nonumber\\
&\overset{\text{(e)}}{=}W_{2,k}+\eta\beta\bsx^\top_k(\bsK+\kappa_1\bsL)\Big(\bm{v}_k+\frac{1}{\beta}\bsg_k^0\Big)\nonumber\\
&~~~+\|\bsx_k\|^2_{\frac{1}{2}\eta^2\beta^2(\bsL+\kappa_1\bsL^2)}
+\frac{1}{2\beta^2}\Big\|\bsg_{k+1}^0-\bsg_{k}^0\Big\|^2_{\bsQ+\kappa_1\bsK}\nonumber\\
&~~~+\frac{1}{\beta}\Big(\bm{v}_k+\frac{1}{\beta}\bsg_{k}^0
+\eta\beta\bsL\bm{x}_k\Big)^\top(\bsQ
+\kappa_1\bsK)(\bsg_{k+1}^0-\bsg_{k}^0)\nonumber\\
&\overset{\text{(f)}}{\le} W_{2,k}+\eta\beta\bsx^\top_k(\bsK+\kappa_1\bsL)\Big(\bm{v}_k+\frac{1}{\beta}\bsg_k^0\Big)\nonumber\\
&~~~+\|\bsx_k\|^2_{\frac{1}{2}\eta^2\beta^2(\bsL+\kappa_1\bsL^2)}
+\frac{1}{2\beta^2}\|\bsg_{k+1}^0-\bsg_{k}^0\|^2_{\bsQ+\kappa_1\bsK}\nonumber\\
&~~~+\frac{\eta}{2}\Big\|\bm{v}_k+\frac{1}{\beta}\bsg_{k}^0\Big\|^2_{\bsQ+\kappa_1\bsK}
+\frac{1}{2\eta\beta^2}\|\bsg_{k+1}^0-\bsg_{k}^0\|^2_{\bsQ+\kappa_1\bsK}\nonumber\\
&~~~+\frac{1}{2}\eta^2\beta^2\|\bsL\bm{x}_k\|^2_{\bsQ+\kappa_1\bsK}
+\frac{1}{2\beta^2}\|\bsg_{k+1}^0-\bsg_{k}^0\|^2_{\bsQ+\kappa_1\bsK}\nonumber\\
&\overset{\text{(g)}}{=} W_{2,k}+\eta\beta\bsx^\top_k(\bsK+\kappa_1\bsL)\Big(\bm{v}_k+\frac{1}{\beta}\bsg_k^0\Big)\nonumber\\
&~~~+\|\bsx_k\|^2_{\eta^2\beta^2(\bsL+\kappa_1\bsL^2)}
+\Big\|\bm{v}_k+\frac{1}{\beta}\bsg_{k}^0
\Big\|^2_{\frac{1}{2}\eta(\bsQ+\kappa_1\bsK)}\nonumber\\
&~~~+\frac{1}{\beta^2}\Big(1+\frac{1}{2\eta}\Big)
\|\bsg_{k+1}^0-\bsg_{k}^0\|^2_{\bsQ+\kappa_1\bsK}\nonumber\\
&\overset{\text{(h)}}{\le} W_{2,k}+\eta\beta\bsx^\top_k(\bsK+\kappa_1\bsL)\Big(\bm{v}_k+\frac{1}{\beta}\bsg_k^0\Big)\nonumber\\
&~~~+\|\bsx_k\|^2_{\eta^2\beta^2(\bsL+\kappa_1\bsL^2)}
+\Big\|\bm{v}_k+\frac{1}{\beta}\bsg_{k}^0
\Big\|^2_{\frac{1}{2}\eta(\bsQ+\kappa_1\bsK)}\nonumber\\
&~~~+\frac{1}{\beta^2}\Big(1+\frac{1}{2\eta}\Big)\Big(\frac{1}{\rho_2(L)}+\kappa_1\Big)
\|\bsg_{k+1}^0-\bsg_{k}^0\|^2\nonumber\\
&\overset{\text{(i)}}{\le} W_{2,k}+\eta\beta\bsx^\top_k(\bsK+\kappa_1\bsL)\Big(\bm{v}_k+\frac{1}{\beta}\bsg_k^0\Big)\nonumber\\
&~~~+\|\bsx_k\|^2_{\eta^2\beta^2(\bsL+\kappa_1\bsL^2)}
+\Big\|\bm{v}_k+\frac{1}{\beta}\bsg_{k}^0
\Big\|^2_{\frac{1}{2}\eta(\bsQ+\kappa_1\bsK)}\nonumber\\
&~~~+\frac{\eta}{\beta^2}\Big(\eta+\frac{1}{2}\Big)
\Big(\frac{1}{\rho_2(L)}+\kappa_1\Big)L_f^2\|\bar{\bsg}^e_{k}\|^2,\label{zerosg:v2k-2}
\end{align}
where (e) holds due to (a) holds due to Lemma~1 and~2 in~\cite{Yi2018distributed}; (f) holds due to the Cauchy--Schwarz inequality; (g) holds due to Lemma~1 and~2 in~\cite{Yi2018distributed}; (h) holds due to $\rho(\bsQ+\kappa_1\bsK)\le\rho(\bsQ)+\kappa_1\rho(\bsK)$, and $\rho(\bsK)=1$; (i) holds due to $\|\bsg^0_{k+1}-\bsg^0_{k}\|^2\le \eta^2L_f^2\|\bar{\bsg}^e_{k}\|^2
\le\eta^2L_f^2\|\bsg^e_{k}\|^2$.

Moreover, we have the following two inequalities hold:
\begin{equation}\label{zerosg:v2k-3}
\|\bsg_{k+1}^0\|^2_{\bsQ+\kappa_1\bsK}
\le\Big(\frac{1}{\rho_2(L)}+\kappa_1\Big)\|\bsg_{k+1}^0\|^2.
\end{equation}

\begin{align}\label{zerosg:v2k-4}
\Big\|\bm{v}_k+\frac{1}{\beta_k}\bsg_{k}^0\Big\|^2_{\bsQ+\kappa_1\bsK}
\le\Big(\frac{1}{\rho_2(L)}+\kappa_1\Big)\Big\|\bm{v}_k+\frac{1}{\beta_k}\bsg_{k}^0
\Big\|^2_{\bsK}.
\end{align}

Then, from \eqref{zerosg:v2k-1}--\eqref{zerosg:v2k-4}, we have
\begin{align}
&W_{2,k+1}\nonumber\\
&\le W_{2,k}
+\eta\beta\bsx^\top_k(\bsK+\kappa_1\bsL)\Big(\bm{v}_k+\frac{1}{\beta}\bsg_k^0\Big)\nonumber\\
&~~~+\frac{1}{2}\eta\Big(\frac{1}{\rho_2(L)}+\kappa_1\Big)
\Big\|\bm{v}_k+\frac{1}{\beta}\bsg_{k}^0\Big\|^2_{\bsK}\nonumber\\
&~~~+\|\bsx_k\|^2_{\eta^2\beta^2(\bsL+\kappa_1\bsL^2)}\nonumber\\
&~~~+\frac{\eta}{\beta^2}\Big(\eta+\frac{1}{2}\Big)
\Big(\frac{1}{\rho_2(L)}+\kappa_1\Big)L_f^2\|\bar{\bsg}^e_{k}\|^2.
\label{zerosg:v2k}
\end{align}

(iii) We have
\begin{align}
W_{3,k+1}=\bsx_{k+1}^\top\bsK\Big(\bm{v}_{k+1}+\frac{1}{\beta}\bsg_{k+1}^0\Big).\label{zerosg:v3k-1}
\end{align}

\begin{align}
&\mathbb{E}\Big[W_{3,k+1}\Big]\nonumber\\
&\overset{\mathrm{Eq.~\ref{zerosg:alg:random-pd-compact}}}{=\joinrel=}\mathbb{E}\Big[(\bm{x}_k-\eta(\alpha\bsL\bm{x}_k+\beta\bm{v}_k
+\bsg_k^0+\bsg^e_k-\bsg_k^0))^\top
\nonumber\\
&~~~\bsK\Big(\bm{v}_k+\frac{1}{\beta}\bsg_{k}^0+\eta\beta\bsL\bm{x}_k
+\frac{1}{\beta}(\bsg_{k+1}^0-\bsg_{k}^0)\Big)\Big]\nonumber\\
&\overset{\text{(j)}}{=}\bm{x}_k^\top(\bsK-\eta(\alpha+\eta\beta^2)\bsL)\Big(\bm{v}_k+\frac{1}{\beta}\bsg_{k}^0\Big)+\|\bm{x}_k\|^2_{\eta\beta(\bsL-\eta\alpha\bsL^2)}\nonumber\\
&~~~+\frac{1}{\beta}\bm{x}_k^\top(\bsK-\eta\alpha\bsL)
\mathbb{E}[\bsg_{k+1}^0-\bsg_{k}^0]-\eta\beta\Big\|\bm{v}_k+\frac{1}{\beta}\bsg_{k}^0\Big\|^2_{\bsK}\nonumber\\
&~~~-\eta\Big(\bm{v}_k+\frac{1}{\beta}\bsg_{k}^0\Big)^\top\bsK
\mathbb{E}[\bsg_{k+1}^0-\bsg_{k}^0]\nonumber\\
&~~~-\eta(\bsg_k^s-\bsg_k^0)^\top
\bsK\Big(\bm{v}_k+\frac{1}{\beta}\bsg_{k}^0+\eta\beta\bsL\bm{x}_k\Big)\nonumber\\
&~~~-\frac{1}{\beta}\mathbb{E}[\eta(\bsg^e_k-\bsg_k^0)^\top
\bsK(\bsg_{k+1}^0-\bsg_{k}^0)]\nonumber\\
&\overset{\text{(k)}}{\le}\bm{x}_k^\top(\bsK-\eta\alpha\bsL)\Big(\bm{v}_k+\frac{1}{\beta}\bsg_{k}^0\Big)
+\frac{1}{2}\eta^2\beta^2\|\bsL\bsx_k\|^2\nonumber\\
&~~~+\frac{1}{2}\eta^2\beta^2\Big\|\bm{v}_k+\frac{1}{\beta}\bsg_{k}^0\Big\|^2_{\bsK}
+\|\bm{x}_k\|^2_{\eta\beta(\bsL-\eta\alpha\bsL^2)}\nonumber\\
&~~~+\frac{1}{2}\eta\|\bm{x}_k\|^2_\bsK
+\frac{1}{2\eta\beta^2}\mathbb{E}[\|\bsg_{k+1}^0-\bsg_{k}^0\|^2\nonumber\\
&~~~+\frac{1}{2}\eta^2\alpha^2\|\bsL\bm{x}_k\|^2
+\frac{1}{2\beta^2}\mathbb{E}[\|\bsg_{k+1}^0-\bsg_{k}^0\|^2]\nonumber\\
&~~~
-\eta\beta\Big\|\bm{v}_k+\frac{1}{\beta}\bsg_{k}^0\Big\|^2_{\bsK}\nonumber\\
&~~~+\frac{1}{2}\eta^2\beta^2\Big\|\bm{v}_k+\frac{1}{\beta}\bsg_{k}^0\Big\|^2_{\bsK}
+\frac{1}{2\beta^2}\mathbb{E}[\|\bsg_{k+1}^0-\bsg_{k}^0\|^2]\nonumber\\
&~~~+\frac{1}{2}\eta\|\bsg_k^s-\bsg_k^0\|^2
+\frac{1}{2}\eta\Big\|\bm{v}_k+\frac{1}{\beta}\bsg_{k}^0\Big\|^2_{\bsK}\nonumber\\
&~~~
+\frac{1}{2}\eta^2\|\bsg_k^s-\bsg_k^0\|^2
+\frac{1}{2}\eta^2\beta^2\|\bsL\bm{x}_k\|^2\nonumber\\
&~~~
+\frac{1}{2}\eta^2\mathbb{E}[\|\bsg^e_k-\bsg_k^0\|^2]
+\frac{1}{2\beta^2}\mathbb{E}[\|\bsg_{k+1}^0-\bsg_{k}^0\|^2]\nonumber\\
&\overset{\text{(l)}}{\le} W_{3,k}
-\eta\alpha\bm{x}_k^\top\bsL\Big(\bm{v}_k+\frac{1}{\beta}\bsg_{k}^0\Big)\nonumber\\
&~~~+\|\bm{x}_k\|^2_{\eta(\beta\bsL+\frac{1}{2}\bsK)
+\eta^2(\frac{1}{2}\alpha^2-\alpha\beta+\beta^2)\bsL^2
+\eta(1+3\eta)L_f^2\bsK}\nonumber\\
&~~~+\eta^2\Big[1 + (\frac{1}{2\eta\beta^2}+\frac{3}{2\beta^2}\Big)L_f^2\Big]\mathbb{E}[\|{\bsg}^e_{k}\|^2]\nonumber\\
&~~~+nL_f^2\eta\Big[\frac{p}{4}+(\frac{p}{4}+2)\eta\Big]\delta^2_k\nonumber\\
&~~~-\Big\|\bm{v}_k+\frac{1}{\beta}\bsg_{k}^0\Big\|^2_{\eta(\beta-\frac{1}{2}
-\eta\beta^2)\bsK}.
\label{zerosg:v3k-2}
\end{align}

where (j) holds since $K_nL=LK_n=L$, $\mathbb{E}[\bsg^e_k] = \bsg^s_k$, and that $x_{i,k}$ and $v_{i,k}$ are independent; (k) holds due to the Cauchy--Schwarz inequality, the Jensen's inequality, and $\rho(\bsK)=1$; (l) holds due to $\|\bsg^s_k-\bsg_k^0\|^2\le 2L_f^2\|\bsx_{k}\|^2_{\bsK}+\frac{np}{2}L_f^2\delta_k^2$ and $\mathbb{E}[\|\bsg_k^0-\bsg^e_k\|^2]\le 4L_f^2\|\bsx_{k}\|^2_{\bsK}+4nL_f^2\delta_k^2+2\mathbb{E}[\|\bsg^e_k\|^2]$..

(iv) We have
\begin{align}
&\mathbb{E}[W_{4,k+1}]=\mathbb{E}[n(f(\bar{x}_{k+1})-f^*)]
=\mathbb{E}[\tilde{f}(\bar{\bsx}_{k+1})-nf^*]\nonumber\\
&=\mathbb{E}[\tilde{f}(\bar{\bsx}_k)-nf^*+\tilde{f}(\bar{\bsx}_{k+1})
-\tilde{f}(\bar{\bsx}_k)]\nonumber\\
&\overset{\text{(m)}}{\le}\mathbb{E}[\tilde{f}(\bar{\bsx}_k)-nf^*
-\eta(\bar{\bsg}_{k}^e)^\top\bsg^0_k
+\frac{1}{2}\eta^2L_f\|\bar{\bsg}_{k}^e\|^2]\nonumber\\
&\overset{\text{(n)}}{=}W_{4,k}
-\eta(\bar{\bsg}_{k}^s)^\top\bar{\bsg}^0_k
+\frac{1}{2}\eta^2L_f\mathbb{E}[\|\bar{\bsg}_{k}^e\|^2]\nonumber\\
&\overset{\text{(o)}}{=}W_{4,k}
-\frac{1}{2}\eta(\bar{\bsg}_{k}^s)^\top(\bar{\bsg}^s_k+\bar{\bsg}^0_k-\bar{\bsg}^s_k)\nonumber\\
&~~~-\frac{1}{2}\eta(\bar{\bsg}^s_{k}-\bar{\bsg}^0_k+\bar{\bsg}^0_k)^\top\bar{\bsg}^0_k
+\frac{1}{2}\eta^2L_f\mathbb{E}[\|\bar{\bsg}_{k}^e\|^2]\nonumber\\
&\overset{\text{(p)}}{\le} W_{4,k}-\frac{1}{4}\eta(\|\bar{\bsg}^s_{k}\|^2
-\|\bar{\bsg}^0_k-\bar{\bsg}^s_k\|^2+\|\bar{\bsg}_{k}^0\|^2\nonumber\\
&~~~-\|\bar{\bsg}^0_k-\bar{\bsg}^s_k\|^2)
+\frac{1}{2}\eta^2L_f\mathbb{E}[\|\bar{\bsg}_{k}^e\|^2]\nonumber\\
&\overset{\text{(q)}}{\le} W_{4,k}-\frac{1}{4}\eta\|\bar{\bsg}^s_{k}\|^2
+\|\bsx_k\|^2_{\eta L_f^2\bsK}\nonumber\\
&~~~+\frac{np}{4}\eta L_f^2\delta^2_k-\frac{1}{4}\eta\|\bar{\bsg}_{k}^0\|^2
+\frac{1}{2}\eta^2L_f\mathbb{E}[\|\bar{\bsg}^e_{k}\|^2],\label{zerosg:v4k}
\end{align}
where (m) holds since that $\tilde{f}$ is smooth; (n) holds due to $\mathbb{E}[\bsg^e_k] = \bsg^s_k$, $x_{i,k}$ and $v_{i,k}$ are independent; (o) holds due to $(\bar{\bsg}_{k}^s)^\top\bsg^0_k=(\bsg_{k}^s)^\top\bsH\bsg^0_k=(\bsg_{k}^s)^\top\bsH\bsH\bsg^0_k
=(\bar{\bsg}_{k}^s)^\top\bar{\bsg}^0_k$; (p) holds due to the Cauchy--Schwarz inequality; and (q) holds due to  $\|\bsg^s_k-\bsg_k^0\|^2\le 2L_f^2\|\bsx_{k}\|^2_{\bsK}+\frac{np}{2}L_f^2\delta_k^2$.

(v) 
Define $W_{k+1} = \sum_{i = 1}^{4}W_{i, k+1}$ and then we have the following inequality holds.

\begin{align}
&\mathbb{E}[W_{k+1}]\nonumber\\
&\le W_{k}-\|\bsx_k\|^2_{\eta\alpha\bsL-\frac{1}{2}\eta\bsK
-\frac{3}{2}\eta^2\alpha^2\bsL^2-\eta(1+5\eta)L_f^2\bsK}\nonumber\\
&~~~+\Big\|\bm{v}_k+\frac{1}{\beta}\bsg_k^0\Big\|^2_{\frac{3}{2}\eta^2\beta^2\bsK}+nL_f^2\eta\Big[\frac{p}{4}+(\frac{p}{4}+4)\eta\Big]\delta_k^2\nonumber\\
&~~~+2\eta^2
\mathbb{E}[\|\bsg^e_k\|^2]+\frac{1}{2}\eta\Big(\frac{1}{\rho_2(L)}+\kappa_1\Big)
\Big\|\bm{v}_k+\frac{1}{\beta}\bsg_{k}^0\Big\|^2_{\bsK}\nonumber\\
&~~~+\|\bsx_k\|^2_{\eta^2\beta^2(\bsL+\kappa_1\bsL^2)}\nonumber\\
&~~~+\frac{\eta}{\beta^2}\Big(\eta+\frac{1}{2}\Big)\Big(\frac{1}{\rho_2(L)}+\kappa_1\Big)L_f^2
\mathbb{E}[\|\bar{\bsg}^e_{k}\|^2]\nonumber\\
&~~~+\|\bm{x}_k\|^2_{\eta(\beta\bsL+\frac{1}{2}\bsK)
+\eta^2(\frac{1}{2}\alpha^2-\alpha\beta+\beta^2)\bsL^2
+\eta(1+3\eta)L_f^2\bsK}\nonumber\\
&~~~+\eta^2\Big[1 + (\frac{1}{2\eta\beta^2}+\frac{3}{2\beta^2}\Big)L_f^2\Big]\mathbb{E}[\|{\bsg}^e_{k}\|^2]\nonumber\\
&~~~+nL_f^2\eta\Big[\frac{p}{4}+(\frac{p}{4}+2)\eta\Big]\delta^2_k\nonumber\\
&~~~-\Big\|\bm{v}_k+\frac{1}{\beta}\bsg_{k}^0\Big\|^2_{\eta(\beta-\frac{1}{2}
-\eta\beta^2)\bsK}-\frac{1}{4}\eta\|\bar{\bsg}^s_{k}\|^2\nonumber\\
&~~~+\|\bsx_k\|^2_{\eta L_f^2\bsK}\nonumber\\
&~~~+\frac{np}{4}\eta L_f^2\delta^2_k-\frac{1}{4}\eta\|\bar{\bsg}_{k}^0\|^2
+\frac{1}{2}\eta^2 L_f\mathbb{E}[\|\bar{\bsg}^e_{k}\|^2]\nonumber\\
&\overset{\text{(r)}}{\le} W_{k}-\|\bsx_k\|^2_{\eta\bsM_{1}-\eta^2\bsM_{2}-b_{1}\bsK}-\Big\|\bm{v}_k+\frac{1}{\beta}\bsg_{k}^0\Big\|^2_{b^0_{2}\bsK}\nonumber\\
&~~~-\eta\Big(\frac{1}{4}-6c_1(p-1)\Big)\|\bar{\bsg}^0_{k}\|^2\nonumber\\
&~~~+\underbrace{c_1\Big[ 6(p-1)\sigma^2_2+(\frac{3}{n_c}+2)p\sigma^2_1 \Big]}_{\mathcal{O}(np)\eta^2}+\underbrace{c_{3}\eta\delta_k^2L_f^2}_{\mathcal{O}(np^2)\eta\delta_k^2},
\label{zerosg:vkLya}
\end{align}
where (r)  holds due to \eqref{zerosg:rand-grad-esti2}, \eqref{zerosg:rand-grad-esti4}, $\alpha=\kappa_1\beta$,  $\eta=\frac{\kappa_2}{\beta}$, and
\begin{align*}
\bsM_{1}&=(\alpha-\beta)\bsL-\Big(1+3L_f^2+\frac{6L_f^4\kappa_1}{\beta^2}(p-1)\Big)\bsK,\\
\bsM_{2}&=\beta^2\bsL+(2\alpha^2+\beta^2)\bsL^2+8L_f^2\bsK\\
&\quad+6(p-1)\Big(3+\frac{1}{2}L_f+\frac{2L_f^{2}}{\beta^2}\kappa_1 + \frac{L_f^2}{2\beta^2}\Big)\bsK,\\
\kappa_3&=\frac{1}{\rho_2(L)}+\kappa_1+1,\\
b^0_{2}&=\frac{1}{2}\eta(2\beta-\kappa_3)-2.5\kappa_2^2,\\
b_{1}
&=6p\kappa_3L_f^4\frac{\eta}{\beta^2}+12p(\kappa_3+1)L_f^4\frac{\eta^2}{\beta^2},\\
c_1&= \Big(3+\frac{1}{2}L_f+\frac{2L_f^{2}}{\beta^2}\kappa_1 + \frac{L_f^2}{2\beta^2}\Big)n\eta^2 + \frac{L_f^2\kappa_1}{\beta^2}n\eta,\\
c_2 &=\frac{3}{4}pn+\eta n(\frac{p}{2} +6) +\frac{p^2L_f^2}{2\beta^2}\kappa_1,\\
c_3 &=c_2 + \Big(c_1 - \frac{L_f^2\kappa_1}{\beta^2}n\eta\Big)p^2\eta.		
\end{align*}

Consider $p\geq 1$, $\alpha=\kappa_1\beta$, $\kappa_1>1$, $\beta$ is large enough, and $\eta=\frac{\kappa_2}{\beta}$, we have
\begin{align}
\eta\bsM_{1}&\ge [(\kappa_1 - 1)\rho_2(L) - 1]\kappa_2\bsK.\label{zerosg:m1-rand-pd}\\
\eta^2\bsM_{2}&\le [\rho(L)+(2\kappa_1^2+1)\rho(L^2)+1]\kappa_2^2\bsK.\label{zerosg:m2-rand-pd}\\
b^0_{2}&\ge\frac{1}{2}(\kappa_2 - 5\kappa_2^2).\label{zerosg:vkLya-b1}
\end{align}

From \eqref{zerosg:vkLya}--\eqref{zerosg:vkLya-b1}, let $\kappa_4 = [(\kappa_1 - 1)\rho_2(L) - 1]\kappa_2- [\rho(L)+(2\kappa_1^2+1)\rho(L^2)+1]\kappa_2^2 $ we know that \eqref{zerosg:sgproof-vkLya2T} holds.

Similar to the way to get \eqref{zerosg:sgproof-vkLya2T}, we have \eqref{zerosg:v4kspeed}.

\end{proof}

\end{document}